\documentclass[a4paper,12pt]{article}
\usepackage{amsmath,amsfonts,amssymb,amsthm,amscd}
\usepackage{pxfonts}
\usepackage{mathrsfs}
\usepackage{color}
\usepackage[all]{xy}
\usepackage{stmaryrd}

\newtheorem{theorem}{Theorem}[section]
\newtheorem{lemma}[theorem]{Lemma}

\newtheorem{corollary}[theorem]{Corollary}
\theoremstyle{definition}

\theoremstyle{remark}

\allowdisplaybreaks
\numberwithin{equation}{section}
\usepackage[colorlinks=true,linkcolor=blue,citecolor=blue,urlcolor=blue]{hyperref}
\begin{document}
\title{The infinite sum of the cubes of reciprocal balancing numbers}
\author{
 Subhasis Panda\\
\small Department of Mathematics, School of Advanced Sciences\\
\small VIT-AP University, Amravati-522241, Andhra Pradesh, India\\
\small {subhasispanda559@gmail.com}\\\\Aditya Kumar Dash\\ 
\small Department of Mathematics and Applied Statistics, School of Applied Sciences \\
\small KIIT Deemed to be University, Bhubaneswar-751024, India\\
\small {adityadash44@gmail.com}\\\\Utkal Keshari Dutta\\ 
\small Department of Mathematics and Applied Statistics, School of Applied Sciences \\
\small KIIT Deemed to be University, Bhubaneswar-751024, India\\
\small {utkal.duttafma@kiit.ac.in}
}
\date{}
\maketitle
\renewcommand{\thefootnote}{}
\footnote{2020 \textit{Mathematics Subject Classification.} 11B39, 11B37, 11D09.}
\footnote{\emph{Key words and phrases}: Balancing numbers, Lucas-balancing numbers, Pell equation, linear recurrence sequences.}
\renewcommand{\thefootnote}{\arabic{footnote}}
\setcounter{footnote}{0}

\begin{abstract}
In this paper, we study the infinite reciprocal cubic sum
$\sum_{k=n}^{\infty}\frac{1}{B_k^3}$ of balancing numbers. We obtain
explicit bounds for this sum and use them to determine the integer part
of its reciprocal.
\end{abstract}

\section{Introduction} \label{sec:1}
A positive integer $B_n$ is called a balancing number if it satisfies the balancing condition
\begin{equation*}
	1+2+\cdots+(B_n-1) = (B_n+1)+\cdots+(B_n+r)
\end{equation*}
for some positive integer $r$, called the balancer \cite{BeheraPanda1999}. Many authors have studied various properties of balancing numbers and their connections with square triangular numbers, Pell equations, and Diophantine equations \cite{TchammouTogbe2021,KaraatliKeskinZhu2014,PandaRay2011,Gozeri2018}.

Reciprocal sums associated with recurrence sequences have received considerable attention. For example, Ohtsuka and Nakamura~\cite{OhtsukaNakamura2008} studied infinite sums involving reciprocals of Fibonacci numbers. Zhang and Wang~\cite{ZhangWang2012} later studied the corresponding infinite sum for Pell numbers. Dutta and Ray~\cite{DuttaRay2020} studied partial finite alternating sums involving reciprocals of balancing and Lucas-balancing numbers. Panda et al.~\cite{Panda2018} studied reciprocal sums of balancing and Lucas-balancing numbers and obtained several bounds for these sums. In particular, for positive integers $n$ and $r\geq3$, they obtained bounds for $\displaystyle \sum_{k=n}^{\infty}\frac{1}{B_k^r}$. However, these bounds do not give the exact integer part of the reciprocal when $r=3$.

Xu and Wang~\cite{Xu2013} extended the result of Zhang and Wang~\cite{ZhangWang2012} to the infinite sum of the cubes of reciprocal Pell numbers. They obtained an explicit formula for the integer part of the reciprocal of the corresponding tail sum. In particular, they obtained the following formula.
\begin{equation*}
	\left[\left(\sum_{k=n}^{\infty}\frac{1}{P_k^3}\right)^{-1}\right] =
	\begin{cases}
		P_n^2P_{n-1}+3P_nP_{n-1}^2
		+\left\lfloor-\dfrac{61}{82}P_n-\dfrac{91}{82}P_{n-1}\right\rfloor,
		& \text{if $n$ is even},\\[6pt]
		P_n^2P_{n-1}+3P_nP_{n-1}^2
		+\left\lceil\dfrac{61}{82}P_n+\dfrac{91}{82}P_{n-1}\right\rceil,
		& \text{if $n$ is odd}.
	\end{cases}
\end{equation*}
Following the result of Xu and Wang~\cite{Xu2013}, we study the corresponding problem for balancing numbers in the cubic case and determine the integer part of the reciprocal of the corresponding infinite sum. We now present our main result in more detail.

Let $\{B_n\}_{n\geq0}$ be the balancing sequence, defined by $B_0=0$, $B_1=1$, and $B_n=6B_{n-1}-B_{n-2},$ for $n\geq2$. For each balancing number $B_n$, $8B_n^2+1$ is a perfect square. Its positive square root is denoted by $C_n$, which is called the corresponding Lucas-balancing number. Thus, the Lucas-balancing sequence $\{C_n\}_{n\geq0}$ satisfies the same recurrence $C_n=6C_{n-1}-C_{n-2}$,  $n\geq2$, with $C_0=1$ and $C_1=3$. The Binet formulas for these sequences are
\begin{equation*}
B_n=\frac{\lambda^n-\mu^n}{\lambda-\mu} \,\, \text{and} \,\, C_n=\frac{\lambda^n+\mu^n}{2},
\end{equation*}
where $\lambda=3+2\sqrt{2}$ and $\mu=3-2\sqrt{2}$.

The connection between balancing numbers and Pell numbers allows us to relate the reciprocal sums of balancing numbers to the results of Zhang and Wang~\cite{ZhangWang2012}. If $\{P_n\}$ denotes the Pell sequence, then $P_{2n}=2B_n$ and $P_{2n-1}=2B_{n-1}+C_{n-1}$. Since $P_{2n-1}=B_n-B_{n-1}$, it follows that $2B_{n-1}+C_{n-1}=B_n-B_{n-1}$. These relations allow us to write the term appearing in the correction term in terms of balancing numbers.

We consider the infinite reciprocal cubic sum $\displaystyle \sum_{k=n}^{\infty}\frac{1}{B_k^3}$. Our aim is to obtain explicit upper and lower bounds for this sum and use them to determine the integer part of its reciprocal. We give an explicit formula for this integer part in the cubic case. Our main result is the following.
\begin{theorem} \label{1.1}
	\normalfont	Let $\{B_n\}$ and $\{C_n\}$ denote the balancing and Lucas-balancing sequences, respectively. Then for every $n\geq 2$, we have
	\begin{equation*}
		\left\lfloor
		\left(\sum_{k=n}^{\infty}\frac{1}{B_k^3}\right)^{-1}
		\right\rfloor
		=
		B_n^3-B_{n-1}^3-
		\left\lceil
		\frac{3(2B_{n-1}+C_{n-1})}{164}
		\right\rceil.	
	\end{equation*}
\end{theorem}
In contrast to the result of Xu and Wang~\cite{Xu2013}, our formula is the same for both even and odd values of $n$. This is because the smaller Binet root of the balancing sequence is positive, namely $\beta=3-2\sqrt{2}$. For the Pell sequence, the corresponding smaller root is negative, which leads to the parity distinction in the result of Xu and Wang~\cite{Xu2013}. In our proof, the irrational correction term can be expressed in the rational form appearing in Theorem~\ref{1.1}. Hence, we obtain an elementary proof similar to the approach of Xu and Wang~\cite{Xu2013}.

\section{Number-theoretic background and consequences} \label{sec:2}
In this section, we collect several identities and auxiliary results for balancing and Lucas-balancing numbers that will be used in the proof of the main theorem.

\begin{lemma}\label{2.1}
	\normalfont	For every integer $n\geq 1$, we have
	\begin{equation*}
		 C_n=B_{n+1}-3B_n.
	\end{equation*}
	Consequently, $2B_{n-1}+C_{n-1}=B_n-B_{n-1}$.
\end{lemma}

\begin{proof}
	Let $D_n=B_{n+1}-3B_n$. Since $B_{n+2}=6B_{n+1}-B_n$, we have
	\begin{align*}
		D_{n+2}
		& =B_{n+3}-3B_{n+2} \\
		& =6B_{n+2}-B_{n+1}-3B_{n+2} \\
		& =6D_{n+1}-D_n.
	\end{align*}
Thus, the sequence $\{D_n\}$ satisfies the same recurrence relation as the sequence $\{C_n\}$. Since $D_0=C_0$ and $D_1=C_1$, it follows that
	$D_n=C_n$ for all $n$. Hence, by uniqueness, $C_n=B_{n+1}-3B_n$. 
Replacing $n$ by $n-1$ in the identity $C_n=B_{n+1}-3B_n$, we get 
$C_{n-1}=B_n-3B_{n-1}$, and hence  $2B_{n-1}+C_{n-1}=B_n-B_{n-1}$.
\end{proof}
 
\begin{lemma} \label{2.2}
\normalfont For all $n\geq 1$, we have
\begin{equation*}
B_n^2-6B_nB_{n-1}+B_{n-1}^2=1.
\end{equation*}
\end{lemma}
 
 \begin{proof}
Let $u_n=B_{n+1}^2-6B_{n+1}B_n+B_n^2$. Now using the recurrence relation $B_{n+2}=6B_{n+1}-B_n$, we obtain 
\begin{equation*}
\begin{aligned}
	u_{n+1}
	&=B_{n+2}^2-6B_{n+2}B_{n+1}+B_{n+1}^2\\
	&=B_n^2-6B_nB_{n+1}+B_{n+1}^2\\
	&=u_n.
\end{aligned}
\end{equation*}
Hence, $u_n$ is constant. Since $u_0=B_1^2-6B_1B_0+B_0^2=1$, we have $u_n=1$ for all $n\geq0$. Therefore $B_{n+1}^2-6B_{n+1}B_n+B_n^2 = 1$ and by replacing $n$ with $n-1$, we get the required result.
 \end{proof}
 An immediate consequence of Lemma~\ref{2.2} is the following exact two-sided bound for consecutive balancing numbers, which will be used repeatedly throughout the proof.

 \begin{corollary}\label{2.3}
 	\normalfont Let $x=B_k$ and $y=B_{k-1}$ for $k\geq1$.	Then we have $x=3y+\sqrt{8y^2+1}$ and Consequently, $x>(3+2\sqrt{2})y$ for all $y\geq1$, and  $x<6y$ for all $y\geq2$.
 \end{corollary}

 \begin{proof}
 From Lemma~\ref{2.2}, we obtain $x^2-6xy+y^2=1$. Now solving this quadratic equation for $x$ and using $x>y$, we take the positive root to obtain $x=3y+\sqrt{8y^2+1}$. As $\sqrt{8y^2+1}>2\sqrt{2}y$ for $y\geq1$, it follows that  $x>(3+2\sqrt{2})y$. Moreover, for $y\geq2$, we have $1<y^2$ and hence $8y^2+1<9y^2$. Therefore, we get $\sqrt{8y^2+1}<3y$, which gives
 $x<6y$.
 \end{proof}
 
We use the same Binet formula method as in \cite[Theorem 11]{Gozeri2018} to obtain the following relations between Pell and balancing numbers.

\begin{lemma}\label{2.4}
\normalfont	Let $P_n$ and $Q_n$ denote the Pell and Pell--Lucas numbers, respectively. Then, for all $n\geq1$, we have
\begin{equation*}
	P_{2n}=2B_n, \,\, P_{2n-1}=2B_{n-1}+C_{n-1} \,\, \text{and} \,\,	Q_{2n}=2C_n.
\end{equation*}
\end{lemma}

\begin{proof}
Let $\alpha_P=1+\sqrt{2}$ and $\beta_P=1-\sqrt{2}$. Since $\alpha_P^2=3+2\sqrt{2}=\alpha$ and
$\beta_P^2=3-2\sqrt{2}=\beta$, the Binet formulas give $P_{2n} \displaystyle
=\frac{\alpha_P^{2n}-\beta_P^{2n}}{2\sqrt{2}}
=\frac{\alpha^n-\beta^n}{2\sqrt{2}}
=2B_n$ and $Q_{2n}
=\alpha_P^{2n}+\beta_P^{2n}
=\alpha^n+\beta^n
=2C_n$. Similarly, for the middle identity, we have
\begin{align*}
	P_{2n-1}
	&=\frac{\alpha_P\alpha^{n-1}-\beta_P\beta^{n-1}}{2\sqrt{2}}\\
	&=\frac{(1+\sqrt{2})\alpha^{n-1}-(1-\sqrt{2})\beta^{n-1}}
	{2\sqrt{2}}\\
	&=\frac{\alpha^{n-1}-\beta^{n-1}
		+\sqrt{2}\left(\alpha^{n-1}+\beta^{n-1}\right)}
	{2\sqrt{2}}\\
	&=\frac{4\sqrt{2}\,B_{n-1}+2\sqrt{2}\,C_{n-1}}
	{2\sqrt{2}}\\
	&=2B_{n-1}+C_{n-1}.
\end{align*}
\end{proof}
Now, using the second identity in Lemma~\ref{2.4}  together with Lemma~\ref{2.1}, we obtain
\begin{equation}\label{e1}
P_{2k-1}=2B_{k-1}+C_{k-1}
=2B_{k-1}+B_k-3B_{k-1}
=B_k-B_{k-1}.
\end{equation}

\section{Auxiliary Results}
We first establish a few auxiliary results that will be used in the proof of the main theorem. The first result gives a rational form of the correction term. This allows us to avoid studying the distribution of an irrational multiple.

\begin{lemma}\label{3.1}
\normalfont For every integer $n\geq 2$,
\begin{equation*}
\left\lceil \frac{3(1+\sqrt{2})}{82}B_{n-1}\right\rceil
=
\left\lceil \frac{3(2B_{n-1}+C_{n-1})}{164}\right\rceil.
\end{equation*}
\end{lemma}

\begin{proof}
Let $\alpha_P=1+\sqrt{2}$ and $\beta_P=1-\sqrt{2}$. Since $  \alpha_P^2=\alpha$ and $\beta_P^2=\beta$, we have $ \displaystyle B_{n-1}=\frac{P_{2n-2}}{2}$. Hence $ \displaystyle \frac{3(1+\sqrt{2})}{82}B_{n-1} = \frac{3\alpha_P P_{2n-2}}{164}$. Now,  using the Binet formula for $P_{2n-2}$, we obtain 
\begin{equation*}
	\begin{aligned}
		\frac{3(1+\sqrt{2})}{82}B_{n-1}
		&=\frac{3\alpha_P}{164}
		\frac{\alpha_P^{2n-2}-\beta_P^{2n-2}}{2\sqrt{2}}\\
		&=\frac{3}{164}
		\left(
		\frac{\alpha_P^{2n-1}-\beta_P^{2n-1}}{2\sqrt{2}}
		-(\sqrt{2}-1)^{2n-2}
		\right)\\
		&=\frac{3P_{2n-1}}{164}
		-\frac{3(\sqrt{2}-1)^{2n-2}}{164}.
	\end{aligned}
\end{equation*}

Set $ \displaystyle \delta_n=\frac{3(\sqrt{2}-1)^{2n-2}}{164}$, and since $0<\sqrt{2}-1<1$, for $n\geq2$ we have 
\begin{equation*}
 	0<\delta_n\leq\frac{3(\sqrt{2}-1)^2}{164}
 	=\frac{3(3-2\sqrt{2})}{164}
 	<\frac{1}{164}.
 \end{equation*}
Since $P_{2n-1}$ is an integer, the fractional part of
$3P_{2n-1}/164$, when it is nonzero, is at least $1/164$.
Therefore, subtracting $\delta_n$, where
$0<\delta_n<1/164$, does not change the ceiling. Thus, we obtain
\begin{equation*}
\left\lceil
\frac{3(1+\sqrt{2})}{82}B_{n-1}
\right\rceil
=
\left\lceil
\frac{3P_{2n-1}}{164}
\right\rceil.
\end{equation*}
Finally, by Lemma~\ref{2.4}, we get $P_{2n-1}=2B_{n-1}+C_{n-1}$. Hence 
\begin{equation*}
	\left\lceil
	\frac{3(1+\sqrt{2})}{82}B_{n-1}
	\right\rceil
	=
	\left\lceil
	\frac{3(2B_{n-1}+C_{n-1})}{164}
	\right\rceil.
\end{equation*}
\end{proof}
The periodicity of balancing numbers modulo positive integers was studied by Panda and Rout \cite{PandaRout2014}. In our setting, we next determine the residues of the correction term modulo $164$. This allows us to write the ceiling term in an exact form that will be useful in the proof of the main theorem.

\begin{lemma}\label{3.2}
\normalfont Let $r_k$ be the least nonnegative residue defined by
\begin{equation*}
r_k\equiv 3(2B_{k-1}+C_{k-1})\pmod{164},
\qquad 0\leq r_k<164,	
\end{equation*}
and put $s_k=164-r_k$. Then $(r_k)_{k\geq2}$ is periodic with period $5$. Moreover,
\begin{equation*}
s_k=
\begin{cases}
	149, & k\equiv 2,4\pmod{5},\\
	77, & k\equiv 3\pmod{5},\\
	161, & k\equiv 0,1\pmod{5}.
\end{cases}
\end{equation*}
\end{lemma}

\begin{proof}
From Lemma~\ref{2.1}, we have $C_{k-1}=B_k-3B_{k-1}$. Therefore, $2B_{k-1}+C_{k-1}=B_k-B_{k-1}$. Let $T_k=B_k-B_{k-1}$. Then, using the recurrence for the balancing numbers, we obtain $T_{k+2}=6T_{k+1}-T_k$. The initial values are $T_2=5,\,T_3=29,\, T_4=169,\,T_5=985,\,T_6=5741$. Thus, modulo $164$, we have
\begin{equation*}
T_2,T_3,T_4,T_5,T_6 \equiv 5,29,5,1,1\pmod{164}.
\end{equation*}
Consequently, we obtain $T_7=6T_6-T_5\equiv5\pmod{164}$, and $T_8=6T_7-T_6\equiv29\pmod{164}$. Hence, $T_7\equiv T_2\pmod{164}$ and $T_8\equiv T_3\pmod{164}$. Since both $(T_k)$ and $(T_{k+5})$ satisfy the same recurrence modulo $164$, it follows that $T_{k+5}\equiv T_k\pmod{164}$ for every $k\geq2$. Thus $(T_k)$ is periodic modulo $164$ with period $5$. Since $r_k\equiv3T_k\pmod{164}$, we obtain $r_k\equiv15,87,15,3,3\pmod{164}$ for $k\equiv2,3,4,0,1\pmod5$, respectively. Therefore,
\begin{equation*}
s_k=164-r_k
=
\begin{cases}
	149, & k\equiv2,4\pmod5,\\
	77, & k\equiv3\pmod5,\\
	161, & k\equiv0,1\pmod5.
\end{cases}
\end{equation*}
\end{proof}
 For $k\ge2$, we define  $E_k:=164 \, (B_k^3-B_{k-1}^3)-3 \, (B_k-B_{k-1})-s_k$ and $D_k:=E_k/164$. Since $B_k-B_{k-1}=2B_{k-1}+C_{k-1}$, by Lemma~\ref{3.2}, $D_k\in\mathbb Z$. Next, the following inequalities will be used to bound the reciprocal tail in the proof of Theorem~\ref{1.1}.
 
 \begin{lemma}\label{3.3}
\normalfont	For every integer $k\ge2$, we have
\begin{itemize}
	\item[(i)] $ \displaystyle 164\,B_k^3\,(E_{k+1}-E_k) >E_k\,E_{k+1}$.	
	\item[(ii)] $\displaystyle	164\,B_k^3\,(E_{k+1}-E_k)<(E_k+164)\,(E_{k+1}+164)$.
\end{itemize}
\end{lemma}

\begin{proof}
Let $x=B_k$, $y=B_{k-1}$, $s=s_k$ and $t=s_{k+1}$. From Corollary~\ref{2.3}, we get
$x>(3+2\sqrt{2})\,y$ for $y\ge1$ and
$x<6y$ for $y\ge2$. Next, by Lemma~\ref{3.2}, $s_k$ has period $5$. Since $t=s_{k+1}$, the possible consecutive pairs $(s,t)$ are $(149,77)\,,(77,149)\,,(149,161)\,,(161,161)\,,(161,149)$, with $77\le s,t\le161$. Now, by Lemma~\ref{2.2} and the recurrence relation, we have $x^2-6xy+y^2=1$ and
$B_{k+1}=6x-y$. Therefore, we get
\begin{align*}
164x^3(E_{k+1}-E_k)-E_kE_{k+1} & =1136520sxy^2-194996sy^3+35409sx+194835sy\\
	& -164ty^3-3tx+3ty+631512xy\\
	& -105252y^2-st+79331.
\end{align*}
For $k\ge3$, we have $y=B_{k-1}\ge2$. Hence, by Corollary~\ref{2.3}, we get $(3+2\sqrt{2})y<x<6y$, and in particular, $x>5y$. Hence
\begin{align*}
	&1136520\,s\,x\,y^2-194996\,s\,y^3-164\,t\,y^3\\
	&\quad>
	\left[77(5\cdot1136520-194996)-164(161)\right]y^3\\
	&\quad=422519104y^3,
\end{align*}	
and $631512\,x\,y-105252\,y^2>3052308y^2$. Moreover, since $t\le161$ and $x<6y$, we have $-3\,t\,x>-3\,(161)\,(6y)=-2898y$. Also, since $s,t\le161$, $-st\ge-161^2=-25921$. Therefore,
\begin{align*}
	&164\,x^3 \,(E_{k+1}-E_k)-E_k\,E_{k+1}\\
	&\quad>422519104\,y^3+3052308\,y^2-2898\,y+53410>0.
\end{align*}
Thus, (i) holds for $k\ge3$. For $k=2$, direct calculation gives  $164\,B_2^3\,(E_3-E_2)-E_2\,E_3=1051418432>0$. Hence (i) is true for $k\ge2$.

Next, for (ii), consider the expansion 
\begin{align*} 
	&(E_k+164)(E_{k+1}+164)-164x^3(E_{k+1}-E_k)\\ 
	={}&(186389280-1136520s)xy^2 +(-32006240+194996s+164t)y^3\\ 
	&-631512xy+105252y^2 +(5806584-35409s+3t)x\\ 
	&+(31953432-194835s-3t)y +st-164s-164t-52435. 
\end{align*}
For each of the possible pairs $(s,t)$ above, direct calculation gives
\begin{equation*}
\begin{array}{|c|c|}
	\hline
	(s,t) &
	5(186389280-1136520s)
	+(-32006240+194996s+164t)\\
	\hline
	(149,77) & 82299792\\
	(77,149) & 477419088\\
	(149,161) & 82313568\\
	(161,161) & 16462320\\
	(161,149) & 16460352 \\
	\hline
\end{array}	
\end{equation*}
Thus, the minimum of these values is $16460352>0$. Since $x>5y$ by Corollary~\ref{2.3},  it follows that
\begin{equation*}
(186389280-1136520s)xy^2 +(-32006240+194996s+164t)y^3 >16460352y^3.
\end{equation*}
Moreover, since $x<6y$ by Corollary~\ref{2.3}, we obtain $-631512xy+105252y^2>-3683820y^2$. Finally, for the remaining terms, for each of the possible pairs $(s,t)$ above, we have $5806584-35409s+3t>0$ and $31953432-194835s-3t>0$. Also, $st-164s-164t-52435\ge -79322$. Therefore, for $k\ge3$, we have 
\begin{align*}
&(E_k+164)(E_{k+1}+164) - 164x^3(E_{k+1}-E_k)\\ &\quad>16460352y^3-3683820y^2-79322.	
\end{align*}
Since $y=B_{k-1} \geq B_2 \geq 6$, the expression on the right is positive. For $k=2$, direct calculation gives $(E_2+164)(E_3+164)-164B_2^3(E_3-E_2) = 101693776>0$. Therefore (ii) is valid for all $k\ge2$.
\end{proof}

\section{Proof of the Main Theorem}

\subsection*{Proof of Theorem \ref{1.1}:} 
\begin{proof}
If we divide (i) by $164^2\,x^3$, we obtain $\frac{E_{k+1}-E_k}{164} > \frac{E_kE_{k+1}}{164^2x^3}$. Now, using $D_k=E_k/164$, this becomes $D_{k+1}-D_k>\frac{D_kD_{k+1}}{x^3}$. Hence 
\begin{equation}\label{4.1}
\frac{1}{x^3} = \frac{1}{B_k^3}
<\frac{1}{D_k}-\frac{1}{D_{k+1}} 
\end{equation}
Similarly, dividing (ii) by $164^2x^3$ gives $\frac{E_{k+1}-E_k}{164} < \frac{(E_k+164)\,(E_{k+1}+164)}{164^2x^3}$. Again, using $D_k=E_k/164$, we obtain $D_{k+1}-D_k < \frac{(D_{k}+1)(D_{k+1}+1)}{x^3}$. Therefore, we have
\begin{equation}\label{4.2}
 \frac{1}{D_{k}+1}-\frac{1}{D_{k+1}+1} < \, \frac{1}{B_k^3} = \frac{1}{x^3}
\end{equation}
Combining \eqref{4.1} and \eqref{4.2}, we obtain
\begin{equation*}
\frac{1}{D_k+1}-\frac{1}{D_{k+1}+1} < \frac{1}{B_k^3} < \frac{1}{D_k}-\frac{1}{D_{k+1}}.	
\end{equation*}
Now, adding the above inequalities for $k=n,n+1,\ldots,m$, we obtain
\begin{equation*}
\sum_{k=n}^{m}
\left(
\frac{1}{D_k+1}-\frac{1}{D_{k+1}+1}
\right)
<
\sum_{k=n}^{m}\frac{1}{B_k^3}
<
\sum_{k=n}^{m}
\left(
\frac{1}{D_k}-\frac{1}{D_{k+1}}
\right).	
\end{equation*}
Hence
\begin{equation*}
\frac{1}{D_n+1}-\frac{1}{D_{m+1}+1}
<
\sum_{k=n}^{m}\frac{1}{B_k^3}
<
\frac{1}{D_n}-\frac{1}{D_{m+1}}.
\end{equation*}
Since $D_m\to\infty$ as $m\to\infty$,  letting $m\to\infty$ gives
\begin{equation*}
\frac{1}{D_n+1} < \sum_{k=n}^{\infty}\frac{1}{B_k^3} < \frac{1}{D_n}.
\end{equation*}
As a result, we have
\begin{equation*}
D_n<
\left(\sum_{k=n}^{\infty}\frac{1}{B_k^3}\right)^{-1}
<D_n+1.	
\end{equation*}
Since $D_n\in\mathbb{Z}$, it follows that $\left\lfloor \left(\sum_{k=n}^{\infty}\frac{1}{B_k^3}\right)^{-1} \right\rfloor=D_n.$ Now, using the definition of $D_n$, we get 
\begin{equation} \label{4.3}
D_n = B_n^3-B_{n-1}^3 -\frac{3(B_n-B_{n-1})+s_n}{164}.	
\end{equation}
By Lemma~\ref{3.2}, we have $s_n=164-r_n$, where $0<r_n<164$.
Since $3(2B_{n-1}+C_{n-1})\equiv r_n\pmod{164}$, we have
\begin{equation*}
\frac{3(2B_{n-1}+C_{n-1})}{164}
=
q+\frac{r_n}{164}
\end{equation*}
for some integer $q$. Hence 
\begin{equation*}
	\left\lceil
	\frac{3(2B_{n-1}+C_{n-1})}{164}
	\right\rceil
	=q+1
	=
	\frac{3(2B_{n-1}+C_{n-1})+164-r_n}{164}
\end{equation*}
and 
\begin{equation*}
\frac{3(2B_{n-1}+C_{n-1})+s_n}{164} = \left\lceil \frac{3(2B_{n-1}+C_{n-1})}{164}\right\rceil.
\end{equation*}
Therefore form equation~\eqref{4.3} and the $B_n-B_{n-1}=2B_{n-1}+C_{n-1}$, we get
\begin{equation*}
D_n = B_n^3-B_{n-1}^3 -\left\lceil \frac{3(2B_{n-1}+C_{n-1})}{164}
\right\rceil. 
\end{equation*}
Since $\left\lfloor
\left(\sum_{k=n}^{\infty}\frac{1}{B_k^3}\right)^{-1}
\right\rfloor=D_n$, we obtain
\begin{equation*}
\left\lfloor
\left(\sum_{k=n}^{\infty}\frac{1}{B_k^3}\right)^{-1}
\right\rfloor
=
B_n^3-B_{n-1}^3
-\left\lceil
\frac{3(2B_{n-1}+C_{n-1})}{164}
\right\rceil.	
\end{equation*}
This completes the proof.
\end{proof}
\section*{Declarations}
\textbf{Conflict of interest:} 
The authors declare that they do not have conflict of interests. \vspace{0.2cm}\\
\textbf{Funding:}
No funding was received.

\end{document}